\documentclass[reqno]{amsart}
\usepackage{url}
\usepackage{amsthm}
\usepackage{amsmath}
\usepackage{amssymb}
\usepackage{graphicx}

\newtheorem{thm}{Theorem}
\newtheorem{cor}{Corollary}

\numberwithin{equation}{section}
\newcommand{\abs}[1]{\lvert#1\rvert}
\newcommand{\gfextn}[0]{pdf}
\DeclareMathOperator{\li}{li}

\begin{document}

\title[Nonnegative trigonometric polynomials]{Nonnegative trigonometric polynomials and a zero-free region for the Riemann zeta-function}
\author{Michael J. Mossinghoff}
\address{Department of Mathematics and Computer Science, Davidson College, Davidson, NC, 28035-6996, USA}
\thanks{The first author was partially supported by a grant from the Simons Foundation (\#210069).}
\email{mimossinghoff@davidson.edu}
\author{Timothy S. Trudgian}
\address{Mathematical Sciences Institute, The Australian National University, ACT 0200, Australia}
\thanks{The second author was partially supported by Australian Research Council DECRA Grant DE120100173.}
\email{timothy.trudgian@anu.edu.au}

\date\today
\subjclass[2010]{Primary: 11M26, 42A05; Secondary: 11Y35}
\keywords{Riemann zeta-function, zero-free region, nonnegative trigonometric polynomials, simulated annealing.}

\begin{abstract}
We prove that the Riemann zeta-function $\zeta(\sigma + it)$ has no zeros in the region $\sigma \geq 1 - 1/(5.573412 \log\abs{t})$ for $\abs{t}\geq 2$.
This represents the largest known zero-free region within the critical strip for $3.06\cdot10^{10} < \abs{t}<\exp(10151.5)$.
Our improvements result from determining some favorable trigonometric polynomials having particular properties, and from analyzing the error term in the method of Kadiri.
We also improve an upper bound in a question of Landau regarding nonnegative trigonometric polynomials.
\end{abstract}

\maketitle

\section{Introduction}\label{sectionIntroduction}
Let $\zeta(s)$ denote the Riemann zeta-function, where throughout this article we write $s=\sigma+it$, with $\sigma$ and $t$ real numbers.
It is well known that $\zeta(s)$ is zero at each negative even integer---these are the trivial zeros of the zeta-function---and that all nontrivial zeros of this function occur in the critical strip in the complex plane, where $0<\sigma<1$.
Further, the nontrivial zeros are symmetric about the line $\sigma=1/2$.
The Riemann hypothesis states that all nontrivial zeros are on this line.

Determining zero-free regions within the critical strip has long been of great interest in number theory, since such results bear directly on questions regarding the distribution of prime numbers.
For example, the prime number theorem was established in 1896 by Hadamard and de la Vall\'ee Poussin by proving that $\zeta(s)$ has no zeros on the line $\sigma=1$.
Moreover, de la Vall\'ee Poussin established in 1899 \cite{dlVP} that $\zeta(s)\neq0$ in a region of the form
\[
\sigma > 1 - \frac{1}{R_0\log\abs{t}},
\]
where $R_0$ is a particular constant and $\abs{t}$ is sufficiently large.
We refer to a region of this form as a classical zero-free region of the Riemann zeta-function.

An improved zero-free region for $\zeta(s)$ was established by Korobov \cite{Korobov} and Vinogradov \cite{Vinogradov} in 1958, who showed independently that $\zeta(s)\neq0$ for
\[
\sigma > 1 - \frac{1}{R_1(\log\abs{t})^{2/3}(\log\log\abs{t})^{1/3}}
\]
for some constant $R_1$.
Ford \cite{Ford2000} established an explicit region of this type, showing that one can take $R_1=57.54$ for $\abs{t}\geq3$.
Clearly, Ford's region is asymptotically superior to the classical one, but for values of $\abs{t}$ that are not too large, explicit forms of the classical region provide better bounds on $\sigma$ than Ford's result.
Because of this, explicit forms of the classical zero-free region remain of interest, and find application in number theory: see for instance \cite{Faber,RamMob}.

Table~\ref{tableR0History} exhibits the history of improvements in the value of $R_{0}$ in the classical region, beginning with de la Vall\'ee Poussin's value, $R_{0}=30.4679$, from 1899, through to the 2005 result of Kadiri \cite{Kadiri}, who established that $R_{0}=5.69693$ is permissible.
Each of these works employed an even trigonometric polynomial $f(\varphi) = \sum_{k=0}^n a_k \cos(k\varphi)$ with real coefficients and possessing certain properties: $f(\varphi)\geq0$ for all real $\varphi$, each coefficient $a_k$ is nonnegative, and $a_1>a_0$.
For example, de la Vall\'ee Poussin used the polynomial $(1+\cos\varphi)^2=\frac{3}{2}+2\cos\varphi+\cos2\varphi$.
After this result, each of the other efforts summarized in this table employed a polynomial of degree $n=4$, except for Kondrat'ev, who employed one of degree $8$, described below.
The trigonometric polynomial used in each of the other results summarized in Table~\ref{tableR0History} is indicated by the values of $b_1$ and $b_2$ in the last two columns: these designate the nonnegative even trigonometric polynomial
$(b_1+\cos\varphi)^2(b_2+\cos\varphi)^2$.

Kondrat'ev's work was motivated by a result of Landau from 1908 \cite{Landau08}.
Let $P_n$ denote the set of even trigonometric polynomials $f_n(\varphi) = \sum_{k=0}^n a_k \cos(k\varphi)$ having nonnegative real coefficients with $a_1>a_0$ and satisfying $f_n(\varphi)\geq0$ for all real $\varphi$.
Define the real number $V_n$ by
\begin{equation}\label{eqnLandauVn}
V_n = \inf_{f_n\in P_n} \frac{f_n(0)-a_0}{(\sqrt{a_1}-\sqrt{a_0})^2}.
\end{equation}
For convenience, we let $A=A(f_n):=f_n(0)-a_0$.
Certainly $V_n$ is decreasing in $n$, so we define
\begin{equation}\label{eqnLandauV}
V = \lim_{n\to\infty} V_n.
\end{equation}
Landau proved that for $\epsilon>0$ one may take
\[
R_0 = \frac{V}{2} + \epsilon
\]
for $\abs{t}>T(\epsilon)$ in the classical zero-free region.
Stechkin \cite{Stech2} improved this in 1970, showing that one may select
\[
R_0 = \frac{V}{2}\left(1-\frac{1}{\sqrt{5}}\right) + \epsilon
\]
for $\abs{t}>T(\epsilon)$, and by using the degree $4$ trigonometric polynomial shown in this entry in Table~\ref{tableR0History}, Stechkin established his value of $R_0=9.65$ for $\abs{t}>12$.

\begin{table}[t]
\centering
\caption{Improvements in the constant $R_{0}$ in the classical region.}\label{tableR0History}
\begin{tabular}{lllll}
&& $R_0$ & $b_1$ & $b_2$\\
\hline
1899 & de la Vall\'ee Poussin \cite{dlVP} & $30.4679$ & $1$ & ---\\
1938 & Westphal \cite{Westphal} & $17.537$ & $1$ & $0.25$\\
1962 & Rosser \& Schoenfeld \cite{RS1} & $17.51631$ & $1$ & $0.3$\\  
1970 & Stechkin \cite{Stech2} & $9.65$ & $0.91$ & $0.28$\\
1975 & Rosser \& Schoenfeld \cite{RS2} & $9.64591$ & $1$ & $0.3$\\
1977 & Kondrat'ev \cite{Kond} & $9.547897$ & \multicolumn{2}{c}{(see Table~\ref{tableKondratev})}\\
2002 & Ford \cite{Ford2000} & $8.463$ & $0.9$ & $0.225$\\
2005 & Kadiri \cite{Kadiri} & $5.69693$ & $0.91$ & $0.265$\\
\hline
\end{tabular}
\end{table}

Kondrat'ev \cite{Kond} employed a computational strategy to search for trigonometric polynomials $f_n\in P_n$ that would produce improved bounds on $V_n$, and thus on $R_0$.
While the details of his method of optimization are not recorded in his 1977 article, his search at $n=8$ produced the polynomial $k_8(\varphi)$, whose normalized coefficients are listed in Table~\ref{tableKondratev}.

\begin{table}[t]
\centering
\caption{Coefficients of the polynomial of Kondrat'ev, $k_8(\varphi)$.}\label{tableKondratev}
\begin{tabular}{ll}
\hline
$a_0$ & $1$\\
$a_1$ & $1.733792817542616$\\
$a_2$ & $1.110484293773627$\\
$a_3$ & $0.4895739485699287$\\
$a_4$ & $0.1180328991868943$\\
$a_5$ & $7.549474144412732\cdot10^{-9}$\\
$a_6$ & $7.994175811779779\cdot10^{-10}$\\
$a_7$ & $0.009253861629263798$\\
$a_8$ & $0.004429241403972788$\\
$A$ & $3.465567070455195$\\
\hline
\end{tabular}
\end{table}

Kadiri \cite{Kadiri} employed a more complicated analysis to obtain the value $R_0=5.69693$.
We review her method in Section~\ref{sectionKadiriMethod}.
In this article, we obtain an improved value for $R_0$ by using Kadiri's method, with two main improvements.
First, we construct an improved trigonometric polynomial of degree $n=16$, using the randomized optimization method of simulated annealing to search for favorable polynomials.
Second, we analyze the error term of this method and obtain some savings by employing additional computations.
We prove the following theorem.

\begin{thm}\label{thmR0}
There are no zeros of $\zeta(\sigma+it)$ for $|t|\geq 2$ and
\[
\sigma > 1 - \frac{1}{5.573412 \log |t|}.
\]
\end{thm}

This represents the largest known zero-free region for the zeta-function within the critical strip for $3.06\cdot10^{10} < \abs{t} < \exp(10151.5)$.

Our work also allows us to improve a bound on Landau's quantity $V$ from \eqref{eqnLandauV}.
Exact values for $V_n$ are known only for $n\leq6$; for the general case, Arestov and Kondrat'ev proved in 1990 \cite{ArKon} that
\begin{equation}\label{eqnAK90}
34.468305 < V < 34.5035864.
\end{equation}
More information on what is known regarding values of $V_n$ and similar quantities may be found in the survey of R\'ev\'esz \cite{Rev}.
In this article, we determine improved values for $V_n$ for a number of values of $n$, and reduce the gap in \eqref{eqnAK90} by approximately $40\%$.

\begin{thm}\label{thmV}
Let $P_n$ denote the set of even trigonometric polynomials $f_n(\varphi)=\sum_{k=0}^n a_k\cos(k\varphi)$ of degree $n$ having each $a_k\geq0$ and $a_1>a_0$, and let
\[
V = \lim_{n\to\infty} \inf_{f_n\in P_n} \frac{f_n(0)-a_0}{(\sqrt{a_1}-\sqrt{a_0})^2}.
\]
Then $34.468305 < V < 34.4889920009$.
\end{thm}

In Section~\ref{sectionKadiriMethod}, we summarize Kadiri's method, and we describe some improvements we make to this method in Section~\ref{sectionImprovements}.
Section~\ref{sectionErrorTerm} provides details on our bounding of the error term in Kadiri's method.
Section~\ref{sectionAnneal} describes our methods for searching for improved trigonometric polynomials by using simulated annealing. 
Section~\ref{sectionResults} reports on the results of our searches, describes the reduction in the gap of permissible values for $V_n$ for a number of integers $n$, and establishes Theorems~\ref{thmR0} and~\ref{thmV}.
It also investigates one possible avenue for further improvements.
Finally, Section~\ref{sectionApplications} briefly discusses some applications of these results.

\section{Kadiri's method}\label{sectionKadiriMethod}

We summarize the method of Kadiri \cite{Kadiri} for reducing the constant $R_0$ in the classical zero-free region of the Riemann zeta-function.
We use the same notation.
This strategy relies on a number of parameters.
First, one requires a trigonometric polynomial $f_n\in P_n$. 
Next, we let $T_0$ be a height to which the Riemann hypothesis has been verified, and we let $R$ be a positive constant for which the classical zero-free region of the Riemann zeta-function has already been established.
Also, we let $r<R$ be a positive constant, and we let $t_0$ denote a constant in $[1, T_0]$ that may be selected later.

We suppose that $\zeta(s)$ has a zero $\rho_0 = \beta_0+i\gamma_0$ with $\gamma_0>0$ which lies just outside an established zero-free region of the zeta-function, so that
\[
1 - \frac{1}{r\log \gamma_{0}}\leq \beta_{0}\leq 1 - \frac{1}{R\log \gamma_{0}}.
\]
We aim to show that $\beta_0$ in fact satisfies a stronger inequality of the form
\[
\beta_{0}\leq 1 - \frac{1}{R_0\log \gamma_{0}}
\]
for some $R_0\in(r,R)$.
The parameter $r$ thus represents a lower bound on what we might achieve.

We now define
\begin{equation}\label{eqnEtaSigmaW}
\begin{split}
\eta & = \frac{1}{r\log\gamma_{0}}, \quad  \sigma = 1 - \frac{1}{R\log(n\gamma_{0} + t_{0})}, \quad w =  \frac{1-\sigma}{\eta},\\
\eta_{0} & = \frac{1}{r\log T_{0}}, \quad  \sigma_{0} = 1 - \frac{1}{R\log(n T_{0} + t_{0})}, \quad
w_{0} = \frac{1-\sigma_{0}}{\eta_{0}},
\end{split}
\end{equation}
and, for any fixed $\theta\in(\pi/2,\pi)$, we define the function $h_\theta(u)$ by
\begin{equation*}\label{eqnHtheta}
\begin{split}
h_{\theta}(u) = \sec^{2}\theta& \bigg\{\sec^{2}\theta\left(\frac{-\theta}{\tan \theta} - \frac{u}{2}\right)\cos(u\tan\theta) - \frac{2\theta}{\tan \theta} - u\\
&\quad- \frac{\sin(2\theta + u\tan \theta)}{\sin2\theta} + 2\left(1 + \frac{\sin(\theta + u\tan\theta)}{\sin\theta}\right)\bigg\}.
\end{split}
\end{equation*}
We then set $g_1(\theta)=h_\theta(0)$, so that
\[
g_1(\theta) = \sec^2(\theta)\left(3-\theta\tan\theta - 3\theta\cot\theta\right),
\]
and we let
\[
d_1(\theta) = -\frac{2\theta}{\tan\theta}
\]
and
\[
m(\theta) = \max\left\{\abs{h''_\theta(u)} : 0\leq u\leq d_1(\theta)\right\}.
\]
Next, we set $\delta$ to be the solution in $[0,1]$ to the equation $\kappa_2(\delta)=\kappa_3(\delta)$, where
\[
\kappa_{2}(\delta) = \kappa_{2}(\delta, \theta) = \frac{g_{1}(\theta)(2\sigma_{0} - 1) - \frac{m(\theta)\eta_{0}^{2}}{2\sigma_{0} -1}}{(1+ 2\delta) g_{1}(\theta)+ \left(\frac{1}{\delta} + \frac{1}{\delta + 2\sigma_{0} - 1}\right)m(\theta)\eta_{0}^{2}}
\]
and
\[
\kappa_{3}(\delta) = \kappa_{3}(\delta, \theta) = \frac{g_{1}(\theta) (2\sigma_{0} - 1) - \frac{m(\theta)\eta_{0}^{2}}{2\sigma_{0} -1}}{\left(\frac{1}{\delta} + \frac{1+ \delta}{(\delta + 2\sigma_{0} - 1)^{2}}\right) g_{1}(\theta) + \left(\frac{1}{\delta^{3}} + \frac{1}{(\delta + 2\sigma_{0} - 1)^{3}}\right)m(\theta)\eta_{0}^{2}},
\]
noting the additional requirements that
\[
\frac{\sqrt{5} -1}{2}\leq \delta\leq 0.866
\]
and
\[
\frac{1}{\delta} + \frac{1}{0.99 + \delta} \leq \frac{1}{\kappa} \leq \frac{1}{\delta^3} + \frac{1}{(1+ \delta)^3}.
\]
In fact, one may select $0\leq\kappa\leq\min\{\kappa_2(\delta), \kappa_3(\delta)\}$, and $\delta\geq\delta'$ if $\delta'$ is the solution in $[0,1]$ of $\kappa_2(\delta)=\kappa_3(\delta)$, but we find that the described selections for $\delta$ and $\kappa$ satisfy the required constraints in our applications.

Finally, let
\[
K(w, \theta) = \int_0^{d_1(\theta)}(a_1 e^{-u} - a_0) h_{\theta}(u) e^{wu}\, du,
\]
where $a_0$ and $a_1$ are the coefficients of the first two terms of the selected trigonometric polynomial $f_n\in P_n$.
It then follows from the arguments given in \cite{Kadiri} that
\begin{equation}\label{eqnKadiriMaster}
R_0 \leq \frac{A g_1(\theta) (1-\kappa)}{2(K(w,\theta) - C(\eta))},
\end{equation}
where $C(\eta)$ is a small error term, and we recall that $A=f_n(0)-a_0$.
Kadiri proved that $K(w,\theta)$ is increasing for $w\geq w_0$ for each fixed $\theta\in(\pi/2,\pi)$, and that $C(\eta)$ is decreasing and nonpositive for $0\leq\eta\leq\eta_0$.
It follows that
\begin{equation}\label{eqnKadiriMainIneq}
R_0 \leq \frac{A g_1(\theta) (1-\kappa)}{2K(w_{0},\theta)}
\end{equation}
for any fixed $\theta\in(\pi/2,\pi)$.
Kadiri selected Rosser and Schoenfeld's value \cite{RS2} for $R$, so $R=9.645908801$, and chose $r=5$, $T_0=3\,330\,657\,430.697$ (as established in \cite{Wed}), $t_0=10$, $\theta=1.848$, and the trigonometric polynomial of degree $n=4$ shown in Table~\ref{tableR0History}.
Kadiri then obtained the value $5.69693$ by using an iterative procedure.

First, one computes an initial improvement in $R_0$ from \eqref{eqnKadiriMainIneq}, and then one replaces $R$ by this $R_0$, and increases the value of $r$ as well, and performs the computation again.
Note that changing $R$ and $r$ in the method alters the values of $\eta_0$, $\sigma_0$, and $w_0$, and hence the values of $\kappa$ and $K(w_0,\theta)$, and so a new value for $R_0$ is obtained from \eqref{eqnKadiriMainIneq}.
The value chosen for $r$ must remain smaller than any improved value for $R_0$.
This procedure is repeated until the $r$ and $R_0$ values lie within a tolerance of $\Delta=10^{-5}$, and then this iteration is repeated, with $r$ reset to $5$ and $R$ set to the latest value of $R_0$.
This outer iteration is then continued until $R_0$ no longer improves by more than $v=5\cdot10^{-6}$. 
Using Kadiri's trigonometric polynomial and the selected values for $T_0$, $t_0$ and $\theta$, after six iterations this method produces the value $R_0=5.696924085\ldots$\,.

\section{Improvements}\label{sectionImprovements}

Our improvement in the value of $R_0$ over Kadiri's result has three main sources.
First, we employ a new trigonometric polynomial.
Second, we save some information in the error term $C(\eta)$ to use in the iterative calculation.
Third, we adjust some of the other parameters in kind, and exploit some more recent knowledge about the zeros of the zeta-function.

We employ the trigonometric polynomial $F_{16}(\varphi)=\sum_{k=0}^{16} a_k \cos(k\varphi)$, whose coefficients are listed in Table~\ref{tableBestForR0} in Section~\ref{sectionResults}.
A graph of this polynomial over $(\pi/2,\pi)$ is shown in Figure~\ref{figureBestK16}.
The method used to construct this polynomial is described in Section~\ref{sectionAnneal}, and this method guarantees that $F_{16}(\varphi)\geq0$ for all $\varphi$.
It is evident from Table~\ref{tableBestForR0} that each $a_k\geq0$ and $a_1>a_0$, so $F_{16}\in V_{16}$.

To obtain some savings from the error term $C(\eta)$, we consider a new parameter $\eta_1$ with $0<\eta_1\leq\eta_0$.
When $\eta\leq\eta_0$, we have $\gamma_0\geq\exp(1/r\eta_1)$ by \eqref{eqnEtaSigmaW}, and since $\log x/\log (n x + t_0)$ is increasing in $x$ for fixed $n$ and $t_0$, it follows that
\[
w = \frac{r\log\gamma_0}{R\log(n\gamma_0+t_0)} \geq \frac{1/R\eta_1}{\log(n\exp(1/r\eta_1)+t_0)} =: w_1.
\]
Thus, for $\eta\leq\eta_1$ we may bound the denominator in \eqref{eqnKadiriMaster} below by $K(w_1,\theta)$.

Since $C(\eta)$ is decreasing on $[0,\eta_0)$, we have $C(\eta)\leq C(\eta_1)$ for $\eta_1\leq\eta\leq\eta_0$, so in this range the denominator in \eqref{eqnKadiriMaster} is bounded below by $K(w_0,\theta)-C(\eta_1)$.
Therefore, we can replace the bound \eqref{eqnKadiriMainIneq} by
\begin{equation}\label{eqnImprovedBound}
R_{0}\leq \frac{A}{2} g_1(\theta)(1-\kappa) \left(\min_{0< \eta_{1} \leq \eta_{0}} \max\left\{\frac{1}{K(w_1, \theta)}, \frac{1}{K(w_{0}, \theta) - C(\eta_{1})}\right\}\right).
\end{equation}
The function $C(\eta)$ is described in detail in Section~\ref{sectionErrorTerm}.

We select $T_0=3.06\cdot10^{10}$ as established in \cite{PlattPi}, $t_0=10^5$, $\theta=1.85573$, $r=5$, and $R=5.7$.
After each iteration, we reset $r$ to the average of its current value and the value of $R_0$ just obtained, and continue this inner process until the distance between these two values is less than $\Delta=10^{-6}$.
The outer iteration is continued until $R_0$ no longer improves by more than $v=5\cdot10^{-7}$.
In addition, in each computation of $R_0$ a binary search is employed to determine a value for $\eta_1\in(0,\eta_0]$: we search for a value of $\eta_1$ in this range where $K(w_1,\theta)$ and $K(w_0,\theta)-C(\eta_1)$ differ by no more than $\epsilon=10^{-3}$.
Table~\ref{tableR0Computation} lists the values of several parameters in the computation at the end of each inner iteration, so just before $r$ is reset to $5$ and $R_0$ is assigned to $R$.
This table exhibits the values for $\eta_0$, $\eta_1$, $\kappa$, and $\delta$ in each iteration as well.
We obtain the value $R_0=5.5734118005\ldots$ after seven rounds; an eighth round produces $R_0=5.57341178\ldots$\,.

We remark that using the new trigonometric polynomial, combined with the strategy of saving some information in the error term, accounts for approximately $93.6\%$ of the improvement in the value of $R_0$ obtained here.
The remaining $6.4\%$ is due to the larger value of $T_0$ that we employ.

\begin{table}[t]
\caption{Values of parameters in successive runs of the outer iteration.
All values are rounded.
}\label{tableR0Computation}
\begin{tabular}{cccccccc}
$R$ & $r$ & $\eta_0\cdot 10^3$ & $\eta_1\cdot10^3$ & $\kappa$ & $\delta$ & $R_0$\\
\hline
$5.7000000$ & $5.58682$ & $7.41347$ & $0.861315$ & $0.440100$ & $0.620251$ & $5.5868212$\\
$5.5868212$ & $5.57486$ & $7.42938$ & $0.876546$ & $0.439964$ & $0.620293$ & $5.5748558$\\
$5.5748558$ & $5.57357$ & $7.43109$ & $0.878187$ & $0.439949$ & $0.620298$ & $5.5735676$\\
$5.5735676$ & $5.57343$ & $7.43128$ & $0.878364$ & $0.439948$ & $0.620298$ & $5.5734286$\\
$5.5734286$ & $5.57341$ & $7.43130$ & $0.878383$ & $0.439948$ & $0.620298$ & $5.5734136$\\
$5.5734136$ & $5.57341$ & $7.43130$ & $0.878385$ & $0.439948$ & $0.620298$ & $5.5734120$\\
$5.5734120$ & $5.57341$ & $7.43130$ & $0.878386$ & $0.439948$ & $0.620298$ & $5.5734118$\\
\end{tabular}
\end{table}

\section{Analyzing the error term}\label{sectionErrorTerm}

We describe the error term $C(\eta)$ from \eqref{eqnKadiriMaster} in some detail, highlighting aspects where our estimates differ from those used in \cite{Kadiri}.
Following Kadiri, we write $C(\eta) = C_{1}(\eta) + C_{2}(\eta) + C_{3}(\eta) + C_{4}(\eta)$, and describe each of these summands in turn.
First, equations (55) and (56) in \cite{Kadiri} produce $C_{1}(\eta) = \eta g_1(\theta)\sum_{k=0}^{n}a_{k} c_{1}(k) $, where
\[
c_{1}(0) = \frac{\kappa-1}{2} \log \pi + \frac{1}{2} \frac{\Gamma'}{\Gamma} \left(\frac{3}{2}\right) - \frac{\kappa}{2} \frac{\Gamma'}{\Gamma} \left(\frac{\sigma_{0} + \delta}{2} + 1\right)
\]
and
\[
c_{1}(k) = \frac{\kappa-1}{2} \log \frac{2\pi}{k} + \frac{1}{2} \min\{ r_{2}(\sigma_{0} + 2, 3, kT_{0}), r_{3}(\sigma_{0} + 2, 3, kT_{0})\}
\]
for $1\leq k\leq n$, where
\begin{equation*}
\begin{split}
r_{2}(x_{0}, x_{1}, y_{0}) &= \frac{1-\kappa}{2} \log\left\{1 + \frac{(x_{1} + \delta)^{2}}{y_{0}^{2}}\right\} + \frac{1}{y_{0}} \left( \tan^{-1} \frac{y_{0}}{x_{1}} + \kappa \tan^{-1} \frac{y_{0}}{x_{1} + \delta}\right),\\
r_{3}(x_{0}, x_{1}, y_{0}) &= \frac{1}{3y_{0}} \left( \frac{1}{x_{0}} + \frac{\kappa}{x_{0} + \delta}\right) + \frac{x_{1}^{2} + \kappa(x_{1} + \delta)^{2}}{2y_{0}^{2}}.
\end{split}
\end{equation*}

Second, from equation (60) in \cite{Kadiri} we have $C_{2}(\eta) = q_{1}\eta + q_{2}\eta^{2} + q_{3}\eta^{3}$, where
\begin{equation*}
\begin{split}
q_{1} &= - \kappa g_1(\theta)\left\{\frac{a_0}{\delta} + \frac{\sigma_{0} - 1 + \delta}{2} \sum_{k=1}^{n} \frac{a_{k}}{(kT_{0})^{2}}\right\},\\
q_{2} &= M^*(-r/R,\theta) \sum_{k=1}^{n} \frac{a_{k}}{(kT_{0})^{2}},\\
q_{3} &= \frac{a_{0} m(\theta)\kappa}{(\sigma_{0} - 1 + \delta)^{3}} + \frac{m(\theta)\kappa}{\sigma_{0} -1 + \delta} \sum_{k=1}^{n} \frac{a_{k}}{(kT_{0})^{2}}.
\end{split}
\end{equation*}
Here, $M^*(z,\theta)$ is an upper bound on
\[
M(z,\theta) =  \int_{0}^{d_{1}(\theta)} |h_{\theta}''(u)| e^{-zu}\, du.
\]
To compute a value for $M^*(z,\theta)$ in an efficient way when $z$ varies and $\theta$ is fixed, we compute each of the integrals
\[
M_k(\theta) = \int_0^{d_1(\theta)} |h_{\theta}''(u)| u^k \, du
\]
once at the start of the computation, and then the fact that $e^y \leq 1 + y + y^2/2 + y^3/3.45$ for $0\leq y\leq 1.91094\ldots$ implies that we may take
\[
M^*(z,\theta) = M_0(\theta) - M_1(\theta) z + \frac{M_2(\theta) z^2}{2} - \frac{M_3(\theta) z^3}{3.45} 
\]
with $z=-r/R$ since $0<r<R$, provided that $d_1(\theta) \leq y_0$.
This last condition is satisfied for $1.8469\ldots < \theta <\pi$, and the values of $\theta$ that we employ lie within this range.
Naturally, we must take care to round $M_0(\theta)$ and $M_2(\theta)$ up at the last decimal place of the calculated precision, and $M_1(\theta)$ and $M_3(\theta)$ down.
(This is similar to the analysis in \cite{Kadiri}, but here we add the $z^3$ term for a sharper estimate.)

Third, equations (52), (53), and (54) in \cite{Kadiri} produce $C_{3}(\eta) = p_{1}\eta + p_{2}\eta^{2} + p_{3}\eta^{3}$, where
\begin{equation*}
\begin{split}
p_{1} &= a_{1} g_1(\theta) \left\{ \left( \frac{1}{\delta} + \frac{1}{\sigma_{0} - \eta_{0} + \delta}\right)\kappa - 1\right\} + M^*(-r/R,\theta)\eta_{0} \sum_{k=0}^{n} a_{k} c_{30}(kT_{0}, t_0),\\
p_{2} &= \frac{(1+ 2\kappa)m(\theta)\eta_{0}}{\sigma_{0} - \frac{1}{2}} \sum_{k=0}^{n} a_{k} c_{30}(kT_{0}, t_0),\\
p_{3} &= a_{1}m(\theta)\left\{ \left(\frac{1}{\delta^{3}} + \frac{1}{(\sigma_{0} - \eta_{0} + \delta)^{3}}\right)\kappa + 1\right\},
\end{split}
\end{equation*}
and where $c_{30}(t,t_0)$ is required to be an upper bound on the sum
\begin{equation}\label{sigma}
\Sigma(t, t_{0}) = \sum_{\substack{\zeta(\beta+i\gamma)=0\\|\gamma| \geq t+ t_{0}}} \frac{1}{(\gamma - t)^{2}}.
\end{equation}
We depart from \cite{Kadiri} here in obtaining an expression for $c_{30}(t,t_0)$: Kadiri employs an explicit bound on the error term for the function $N(T)$, the number of nontrivial zeros of the Riemann zeta-function with imaginary part in $[0,T]$.
We pursue another approach.
For $t=0$, we have
\[
\Sigma(0,t_0) = 2\left(\sum_{\gamma>0} \frac{1}{\gamma^2} - \sum_{0<\gamma\leq t_0} \frac{1}{\gamma^2}\right),
\]
where all the sums are over the pertinent zeros of the zeta-function in the critical strip.
From \cite[Lemma 2.9]{STDPiLi} we have that the first sum in this last expression is at most $0.023105$, and we may compute the latter sum explicitly for a reasonable value of $t_0$.
Using this strategy, we compute that we may take $c_{30}(0, 10^5)=0.00027$.
For $t>0$, we note that
\[
\Sigma(t) = \sum_{\gamma \geq t + t_{0}} \frac{1}{(\gamma - t)^{2}} + \sum_{\gamma\geq t + t_{0}} \frac{1}{(\gamma + t)^{2}},
\]
and we use a result of Lehman \cite{LehmanPiLi}, which allows us to write
\[
\sum_{\gamma> T_{1}} \phi(\gamma) = \frac{1}{2\pi} \int_{T_{1}}^{\infty} \phi(x) \log (x/2\pi)\, dx + \xi \left\{ 4 \phi(T_{1}) \log T_{1} + 2 \int_{T_{1}}^{\infty} \frac{\phi(x)}{x}\, dx\right\},
\]
where $\xi$ denotes a constant of absolute value at most $1$, when $\phi(x)$ is continuous, positive, and monotone decreasing on $x>T_1>2\pi e$.
Using $\phi(x)=1/(x-t)^2 + 1/(x+t)^2$, and ignoring negligible savings, we find that we may take
\begin{equation*}
\begin{split}
c_{30}(kT_0) &= \frac{1}{2\pi} \int_{t_{0}}^{\infty} \log\left(\frac{x+kT_0}{2\pi}\right) \left( \frac{1}{x^{2}} + \frac{1}{(x+2kT_0)^{2}}\right)\, dx\\
&\qquad + 4 \log(kT_0 + t_{0})\left( \frac{1}{t_{0}^{2}} + \frac{1}{(t_{0} + 2kT_0)^{2}}\right) + \frac{4}{kT_0 t_{0}}
\end{split}
\end{equation*}
for $1\leq k\leq n$.

Last, equations (61) and (64) in \cite{Kadiri} yield
\[
C_{4}(\eta) = \eta^{3}m(\theta)\sum_{k=0}^{n} a_{k} \left\{ C_{41}( k) + C_{42}(k)\right\},
\]
where
\begin{equation*}
\begin{split}
C_{41}(k) &= \frac{1}{2\pi (\sigma_{0} - \frac{1}{2})} \int_{-\infty}^{\infty} \frac{U_{0}(t)}{(\sigma_{0} - \frac{1}{2})^{2} + (kT_{0} - t)^{2}}\, dt\\
&\qquad+ \frac{\kappa}{2\pi (\sigma_{0} - \frac{1}{2} + \delta)} \int_{-\infty}^{\infty} \frac{U_{0}(t)}{(\sigma_{0} - \frac{1}{2} + \delta)^{2} + (kT_{0} - t)^{2}}\, dt,\\
C_{42}(0) &= \frac{1}{\sigma_{0}^{3}} + \frac{\kappa}{(\sigma_{0} + \delta)^{3}}, \quad
C_{42}(k) = \left(\frac{1}{\sigma_{0}} + \frac{\kappa}{\sigma_{0} + \delta}\right) \frac{1}{(kT_{0})^{2}}, \quad \textrm{if } k\geq 1,
\end{split}
\end{equation*}
and, by Lemma 3.6 of \cite{Kadiri},
\begin{equation}
U_{0}(t) \leq \begin{cases} \frac{1}{2} \log \left( \frac{16}{1+ 4t^{2}}\right) + \frac{2}{1+ 4t^{2}} + 2, \quad & \mbox{if } |t|< \frac{1}{2},\\
\bigl| \log \frac{|t|}{2} - \frac{2}{1+ 4t^{2}}\bigr| + \frac{2}{3|t|} + \frac{1}{8t^{2}}, \quad & \mbox{if } |t|\geq \frac{1}{2}.
\end{cases}
\end{equation}
We estimate the first integral in the definition of $C_{41}(k)$ as follows; the second integral is estimated similarly.
We write this integral as
\[
I=\int_{-\infty}^{\infty} \frac{U_{0}(t)}{a^{2} + (b-t)^{2}}\, dt,
\]
where $a = \sigma_{0} - \frac{1}{2}$ and $b = kT_{0}$.
Note that $\log \frac{|t|}{2} - \frac{2}{1+ 4t^{2}}\geq0$ for $|t|\geq t^{*} = 2.205\ldots$\,.
When $k=0$ we make use of the symmetry inherent in the integrand to write
\begin{equation*}
\int_{-\infty}^\infty \frac{U_0(t)}{a^2+t^2}\,dt =
2\left(\int_{0}^{1/2} \frac{U_0(t)}{a^2+t^2}\,dt +
\int_{1/2}^{t^*} \frac{U_0(t)}{a^2+t^2}\,dt +
\int_{t^*}^{\infty} \frac{U_0(t)}{a^2+t^2}\,dt\right).
\end{equation*}
When $k\neq 0$ we write
\begin{equation*}
\left(\int_{-\infty}^{-t^{*}} + \int_{t^{*}}^{-\frac{1}{2}}+ \int_{-\frac{1}{2}}^{-\frac{1}{2}}+ \int_{\frac{1}{2}}^{t^{*}}+ \int_{t^{*}}^{\infty}\right) \frac{U_{0}(t)}{a^{2} + (b-t)^{2}}\, dt = I_{1} + I_{2} + I_{3} + I_{4} + I_{5},
\end{equation*}
say.
We may estimate each of $I_2$, $I_3$, and $I_4$ numerically, noting that the formula used for $U_{0}(t)$ is different in each of these integrals. 
For $|t|\geq t^*$, we find that
\[
U_0(|t|) \leq \log|t| - c,
\]
and the optimal value of $c$ is
\[
c=\log2 + \frac{2}{1+4{t^*}^2}-\frac{2}{3t^*}-\frac{1}{8{t^*}^2} = 0.462935\ldots.
\]
Thus, we obtain
\begin{align*}
I_{1} =\int_{-\infty}^{-t^*} \frac{U_0(t)}{a^2+(b-t)^2}\,dt
&< \int_{-\infty}^{-t^*} \frac{\log(-t)-c}{(b-t)^2}\,dt\\
&= \frac{\log(b+t^*)}{b} - \frac{t^*\log t^*}{b(b+t^*)} - \frac{c}{b+t^*}.
\end{align*}
All that remains is to estimate $I_{5}$.
First, we note that
\[
\int_{t^*}^\infty \frac{c}{a^2+(b-t)^2}\,dt = \frac{c}{a}\left(\tan^{-1}\left(\frac{b-t^*}{a}\right)+\frac{\pi}{2}\right).
\]
We then divide the range into three sections: $[t^{*}, b(1-\varepsilon)]$, $[b(1-\varepsilon), b(1+\varepsilon)]$, and $[b(1+\varepsilon), \infty)$, for a parameter $\varepsilon>0$.
We compute
\begin{align*}
\int_{t^*}^{b(1-\varepsilon)} &\frac{\log t}{a^2+(b-t)^2}\,dt
< \int_{t^*}^{b(1-\varepsilon)} \frac{\log t}{(b-t)^2}\,dt\\
&\qquad= \frac{1}{b}\left(\frac{\log b}{\varepsilon} - \log(b-t^*) + \log\varepsilon - \left(1-\frac{1}{\varepsilon}\right)\log(1-\varepsilon) - \frac{t^*\log t^*}{b-t^*}\right),
\end{align*}
and then
\[
\int_{b(1-\varepsilon)}^{b(1+\varepsilon)} \frac{\log t}{a^2+(b-t)^2}\,dt
 < \log(b(1+\varepsilon))\cdot\frac{2}{a}\tan^{-1}\left(\frac{b\varepsilon}{a}\right) < \frac{\pi}{a}\log(b(1+\varepsilon)),
\]
and finally
\[
\int_{b(1+\varepsilon)}^{\infty} \frac{\log t}{a^2+(b-t)^2}\,dt
< \int_{b(1+\varepsilon)}^{\infty}\frac{\log t}{(b-t)^2}\,dt
= \frac{1}{b\varepsilon}\left(\log b + \varepsilon\log\left(1+\frac{1}{\varepsilon}\right) + \log(1+\varepsilon)\right).
\]
Therefore, we find that
\begin{align*}
&\int_{-\infty}^\infty \frac{U_0(t)}{a^2+(b-t)^2}\,dt
< \int_{-t^*}^{-\frac{1}{2}} \frac{U_0(t)}{a^2+(b-t)^2}\,dt+
\int_{-\frac{1}{2}}^{\frac{1}{2}} \frac{U_0(t)}{a^2+(b-t)^2}\,dt\\
&\qquad + \int_{\frac{1}{2}}^{t^*} \frac{U_0(t)}{a^2+(b-t)^2}\,dt + \frac{\pi}{a} \log(b(1+\varepsilon))\\
&\qquad + \frac{1}{b} \left(\frac{2\log b}{\varepsilon} + \left(1+\frac{1}{\varepsilon}\right)\log(1+\varepsilon) - \left(1-\frac{1}{\varepsilon}\right)\log(1-\varepsilon) + \log\left(\frac{b+ t^*}{b-t^*}\right)\right)\\
&\qquad -\frac{t^*\log t^*}{b}\left(\frac{1}{b+t^*}+\frac{1}{b-t^*}\right) - c\left(\frac{1}{b+t^*} + \frac{1}{a}\left(\tan^{-1}\left(\frac{b-t^*}{a}\right)+\frac{\pi}{2}\right)\right).
\end{align*}
We choose $\varepsilon=10^{-3}$, and thus find that we can bound the term $C_{41}(k)$ relatively easily.

\section{Searching for trigonometric polynomials}\label{sectionAnneal}

We employ simulated annealing to search for favorable trigonometric polynomials in this problem.
In this optimization method, one seeks to minimize an objective function $G$ defined on some (typically multidimensional) domain $D$.
One begins at a particular location, and moves through the space in search of favorable values of $G$ by using an iterative process.
At each step, one moves from the current location to a nearby point in the domain $D$, and measures the change in the value of $G$.
If the objective function decreased, i.e.\ we witnessed an improvement, then we keep this step and proceed to the next iteration.
On the other hand, if we find that $G$ increased due to this move, then we may undo this step, but we may choose to keep it, depending on a particular probability distribution.
This distribution depends on a parameter called the \textit{temperature}: if the change $\Delta G$ in the objective function is positive, then we keep this step with probability $\exp(-\Delta G/Y)$, where $Y$ is the current temperature.
Thus, if the temperature is high, then we are reasonably likely to keep steps that increase the objective function by a small amount, but less likely if we witness a large change.
As the temperature decreases, we become less likely to accept small increases in the value of $G$.
In the limiting case, as $Y\to0^+$, the method becomes simple greedy descent.
Typically, one performs a number of trials for each value in a decreasing sequence of temperatures (the \textit{annealing schedule}), ending with a number of trials of simple greedy descent.

Simulated annealing requires many evaluations of the objective function, so it is very helpful if this function can be computed quickly.
In this problem, where the domain $D$ is the set $P_n$ of trigonometric polynomials of degree $n$ having the required properties, using Kadiri's method to evaluate the merit of a particular polynomial in this space is rather time-consuming.
Instead, we use Landau's function from \eqref{eqnLandauVn}, so for $f\in P_n$ we let
\begin{equation}\label{eqnLandauObjFcn}
G(f) = \frac{A}{a_0+a_1-2\sqrt{a_0 a_1}},
\end{equation}
where $A$ again denotes $f(0)-a_0$.
Note that writing the denominator in this way allows $G$ to be evaluated with only one call to a square-root function.

In our method, we fix the degree $n$, we set $c_0=1$, and then choose each $c_k$ for $1\leq k\leq n$ uniformly at random from $[0,B]$ for  a prescribed bound $B$, and then set
\[
f(\varphi) = \Bigl\lvert\sum_{k=0}^n c_k e^{i k \varphi}\Bigr\rvert^2.
\]
We then compute values $a_0$, \ldots, $a_n$ so that $f(\varphi)=\sum_{k=0}^n a_k \cos(k\varphi)$.
Here,
\[
a_0 = \sum_{j=0}^n c_j^2,\;\;
a_k = 2\sum_{j=0}^{n-k} c_j c_{j+k}
\textrm{\ for $1\leq k\leq n$}.
\]
We note that $a_0$ is the peak autocorrelation of the sequence $\{c_k\}$, and $a_k/2$ for $k>1$ is the $k$th off-peak aperiodic autocorrelation of this sequence.
(In fact, it is slightly more efficient to keep track of these autocorrelations in the code, rather than the $a_k$, but this is a minor point.)
We then test if $f\in P_n$, so we test if $a_1>a_0$ and whether each $a_k\geq0$, since $f(\varphi)\geq0$ for all $\varphi$ automatically by the construction.
(Kondrat'ev \cite{Kond} employed this same strategy to ensure choosing trigonometric polynomials that are nonnegative everywhere.)
If this test fails, then we start over with a new randomly selected $f$.
Once we have a viable polynomial, we proceed with the annealing phase.

To describe the annealing process, we find it convenient to set $Z=1/Y$, where $Y$ is the temperature.
We begin with a particular maximum step size $S$ and an initial value for $Z$.
We then perform $M$ annealing trials: each time we randomly select an integer $k\in\{1,\ldots,n\}$, we pick a step size $s\in[-S,S]$ uniformly at random, and then we change $c_k$ by $s$.
We then update each coefficient $a_j$ that depends on $c_k$, so $a_0$ increases by $s(2c_k+s)$, $a_k$ increases by $2s$, each $a_i$ with $1\leq i<k$ increases by $2sc_{k-i}$, and each $a_i$ with $1\leq i<n-k$ increases by $2sc_{k+i}$.
Thus, each change to a value of $c_k$ requires $O(n)$ operations to update the state of the program.
We then check if the new values of $a_i$ are all nonnegative, and that $a_1>a_0$.
If any of these conditions fails, we undo this step and return to the prior polynomial $f$.
Otherwise, our new location is valid, so we compute the value of the objective function \eqref{eqnLandauObjFcn}.
If $\Delta G<0$, we keep this step, and if $\Delta G>0$, then we keep this step with probability $e^{-Z\Delta G}$.

For each fixed value of the maximum step size $S$, we perform this process $M$ times for each of the $K$ values of $Z$, incrementing $Z$ each time by an amount $\Delta Z$, and then we perform $M$ trials of greedy descent.
After this, we decrease the step size by dividing $S$ by $1+\lambda$, where $\lambda$ is a positive parameter, and continue as long as this maximum size is greater than a prescribed value, $S_0$.

Our selection of parameters in this process varied over a number of runs.
In fact, it was useful to run annealing jobs on many processors simultaneously, with each processor choosing its annealing parameters at random from some prescribed intervals.
This way, over time we could learn which parameters tended to produce better trigonometric polynomials, and this would inform choices of parameters in later runs.
We often chose $B\in[100,200]$, $Z\in[8,16]$, $\Delta Z\in[.5,2]$, $M\in[250n,350n]$, $K=10$ or $11$, $S\in[2.5,4]$, $\lambda\in[.015,.05]$, and $S_0=10^{-5}$ or $10^{-6}$.

\section{Results}\label{sectionResults}

Table~\ref{tableBestVn} records the best value of the objective function we constructed by using this optimization method for each degree $n=4m$, with $2\leq m\leq10$.
Each entry provides the current best-known value for Landau's $V_n$ \eqref{eqnLandauVn}.
In particular, the entry for $n=8$ improves the result of \cite{Kond}, and each subsequent entry is less than the formerly best known upper bound on $V$ of $34.5035\ldots$ from \cite{ArKon}.
At degree $n=40$, our searches produced a polynomial that established $V_{40}<34.48923$.
We then performed some additional simulated annealing passes with finer parameters, using this polynomial as our starting point.
In this way we constructed a trigonometric polynomial $F_{40}$, whose objective value produces $V_{40}<34.488992000856$.
This establishes Theorem~\ref{thmV}.
The coefficient sequence $\{c_k\}$ for $F_{40}(\varphi)$ is shown in Figure~\ref{figureF40Ck}, and the corresponding values for the sequence $\{a_k\}$, after normalizing by dividing by the sum of the squares of the $c_k$ (approximately $268\,761.1$), are displayed in Figure~\ref{figureF40Ak}.
The sum of all the coefficients after $a_0$ in this list is $A=3.490002852278399$.
A plot for $F_{40}(\varphi)$ over $[\pi/2,\pi]$ is exhibited in Figure~\ref{figureBestV40}.

\begin{figure}[t]
\caption{Coefficients $\{c_k\}$ for $F_{40}(\varphi)$.}\label{figureF40Ck}
{\tiny
\begin{equation*}
\begin{split}
\{1, \; &5.82299804516981, 20.256046857268, 41.8371543572416, 62.6803646661328,\\
&86.3371461181984, 140.237920299854, 230.172136828422, 296.566863709684,\\
&259.103949548556, 115.007833826561, -36.3610468402636, -95.0264057388131,\\
&{-58.142464967588}, -2.7847512808475, 17.2100840100492, 10.7193028056306,\\
&5.04620845731704, 3.86024903930945, -0.321719428284046, -4.25534268200327,\\
&{-2.73295949107489}, 1.31013188826865, 1.16950963394944, -0.36353468679768,\\
&{-0.468164164916908}, 0.940864822848766, 0.224723973462492, -0.762490754252658,\\
&{-0.100627491842387}, 0.486777412937381, -0.00593275691352972, -0.399968301710688,\\
&0.225086077352437, 0.198770260275958, -0.239235383947936, -0.0483116501842605,\\
&0.212556439828151, -0.124475686268429, 0.0123641022008911, 0.0146979855952472\}
\end{split}
\end{equation*}
}%
\end{figure}

\begin{figure}[t]
\caption{Normalized coefficients $\{a_k\}$ for $F_{40}(\varphi)$.}\label{figureF40Ak}
{\tiny
\begin{equation*}
\begin{split}
\{1, \; &1.737404932358421, 1.1180312174988238, 0.4958068290777618, 0.12043435038423943,\\
&2.8862586376137162\cdot10^{-15}, 2.0972266792962233\cdot10^{-7}, 0.010264161325212823,\\
&0.005445447694684562, 4.982862012828773\cdot10^{-14}, 4.8092823793119544\cdot10^{-8},\\
&6.247326172856939\cdot10^{-4}, 1.625564268382144\cdot10^{-12}, 2.835025226384611\cdot10^{-7},\\
&6.129588083976596\cdot10^{-4}, 5.231422915759923\cdot10^{-4}, 4.936509000300511\cdot10^{-8},\\
&1.4092336429708759\cdot10^{-13}, 2.923444923088646\cdot10^{-4}, 2.3484834215069657\cdot10^{-4},\\
&1.1342395284997809\cdot10^{-7}, 9.170949177475853\cdot10^{-15}, 1.1678447544116879\cdot10^{-4},\\
&8.482954015115613\cdot10^{-5}, 5.338369918854597\cdot10^{-12}, 1.8115431364172399\cdot10^{-6},\\
&2.274822660800113\cdot10^{-5}, 6.894175279654467\cdot10^{-14}, 1.570310005695605\cdot10^{-7},\\
&3.090195669976266\cdot10^{-5}, 2.75304675142816\cdot10^{-5}, 4.959939074564098\cdot10^{-8},\\
&2.6604018245602642\cdot10^{-14}, 1.643513306111948\cdot10^{-5}, 1.522818206748077\cdot10^{-5},\\
&4.623133724479399\cdot10^{-6}, 7.930664147911865\cdot10^{-7}, 2.6276458478038436\cdot10^{-6},\\
&1.8249928742340992\cdot10^{-6}, 7.289034419080702\cdot10^{-7}, 1.0937584248561976\cdot10^{-7}\}
\end{split}
\end{equation*}
}%
\end{figure}

\begin{table}[t]
\caption{Upper bound on $V_n$ for $n=4m$ and $2\leq m\leq10$, and bound on $R_0$ resulting from choosing this polynomial.}\label{tableBestVn}
\begin{tabular}{clcc}
$n$ & \multicolumn{1}{c}{$V_n <$} & $\theta$ & $R_0 <$\\
\hline
$8$  & $34.53991919$ & $1.853$ & $5.58139$\\
$12$ & $34.50266054$ & $1.855$ & $5.57429$\\
$16$ & $34.49747009$ & $1.855$ & $5.57490$\\
$20$ & $34.49321564$ & $1.855$ & $5.57495$\\
$24$ & $34.49027559$ & $1.855$ & $5.57519$\\
$28$ & $34.48959029$ & $1.855$ & $5.57560$\\
$32$ & $34.48939230$ & $1.855$ & $5.57632$\\
$36$ & $34.48930967$ & $1.855$ & $5.57683$\\
$40$ & $34.488992001$ & $1.855$ & $5.57724$\\ 
\hline
\end{tabular}
\end{table}

\begin{figure}[t]
\caption{$F_{40}(\varphi)$ for $\pi/2\leq\varphi\leq\pi$.}\label{figureBestV40}
\includegraphics[width=4in]{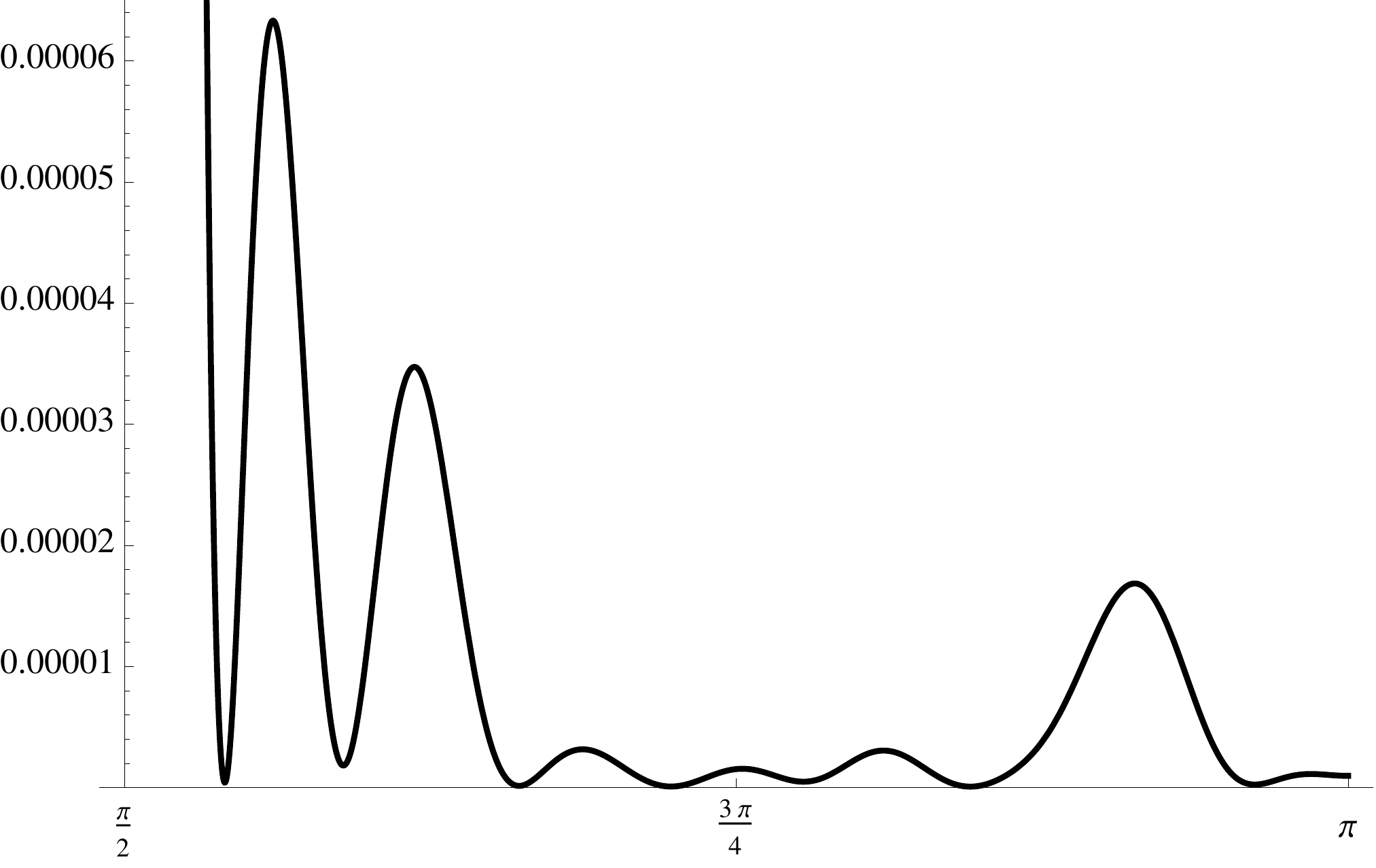}
\end{figure}

We may now apply Kadiri's method for computing a value for $R_0$ in the classical zero-free region of the Riemann zeta-function using each of the polynomials that produced the records for $V_n$, and employing the savings obtained in the error term as described in Section~\ref{sectionImprovements}.
These values are also listed in Table~\ref{tableBestVn}, along with the value of $\theta$ that was employed in each case.
The best value here appears at $n=12$, and some further local optimizations allow us to derive $R_0 < 5.57422$.
It is interesting that polynomials that produce an improvement in Landau's value $V$ do not necessarily produce an improvement in $R_0$ under Kadiri's method, but this is perhaps not so surprising since we use Landau's function as a surrogate for our true objective function, owing to its easy calculation.
This suggests that we test the several hundred thousand other trigonometric polynomials that were produced by our method of simulated annealing, which exhibited very good values under Landau's objective function.
Indeed, many of these produce a value for $R_0$ better than $5.57422$.
The best polynomial found in these searches has degree $16$, and produces $R_0 < 5.57353$.
Some further annealing using this polynomial as a starting point produces the polynomial $F_{16}(\varphi)$ employed in Section~\ref{sectionImprovements} to establish the bound in Theorem~\ref{thmR0}.
The sequences of coefficients $c_k$ and $a_k$ are both displayed in Table~\ref{tableBestForR0}, where again we have normalized the $a_k$ by dividing by the sum of the squares of the $c_k$ (approximately $932\,023.9$), so that $a_0=1$.

We remark that Landau's objective function \eqref{eqnLandauObjFcn} evaluated on $F_{16}(\varphi)$ produces $34.49997\ldots$, which is somewhat larger than the record obtained for this degree, as shown in Table~\ref{tableBestVn}.

\begin{table}[t]
\caption{Coefficients for $F_{16}(\varphi)$.}\label{tableBestForR0}
\begin{tabular}{cl||cl}
\hline
$c_0$ & $1$ & $a_0$ & $1$\\
$c_1$ & $-2.09100370089199$ & $a_1$ & $1.74126664022806$\\
$c_2$ & $0.414661861733616$ & $a_2$ & $1.128282822804652$\\
$c_3$ & $4.94973437766435$ & $a_3$ & $0.5065272432186642$\\
$c_4$ & $-2.26052224951171$ & $a_4$ & $0.1253566902628852$\\
$c_5$ & $-8.58599241204357$ & $a_5$ & $9.35696526707405\cdot10^{-13}$\\
$c_6$ & $6.87053689828658$ & $a_6$ & $4.546614790384321\cdot10^{-13}$\\
$c_7$ & $22.6412990090005$ & $a_7$ & $0.01201214561729989$\\
$c_8$ & $-6.76222005424994$ & $a_8$ & $0.006875849760911001$\\
$c_9$ & $-50.2233943767588$ & $a_9$ & $7.77030543093611\cdot10^{-12}$\\
$c_{10}$ & $8.07550113395201$ & $a_{10}$ & $2.846662294985367\cdot10^{-7}$\\
$c_{11}$ & $223.771572768515$ & $a_{11}$ & $0.001608306592372963$\\
$c_{12}$ & $487.278135806977$ & $a_{12}$ & $0.001017994683287104$\\
$c_{13}$ & $597.268928658734$ & $a_{13}$ & $2.838909054508971\cdot10^{-7}$\\
$c_{14}$ & $473.937203439807$ & $a_{14}$ & $5.482482041999887\cdot10^{-6}$\\
$c_{15}$ & $237.271715181426$ & $a_{15}$ & $2.412958794855076\cdot10^{-4}$\\
$c_{16}$ & $59.6961898512813$ & $a_{16}$ & $1.281001290654868\cdot10^{-4}$\\
& & $A$ & $3.523323140225021$\\
\hline
\end{tabular}
\end{table}

\begin{figure}[t]
\caption{$F_{16}(\varphi)$ for $\pi/2\leq\varphi\leq\pi$.}\label{figureBestK16}
\includegraphics[width=4in]{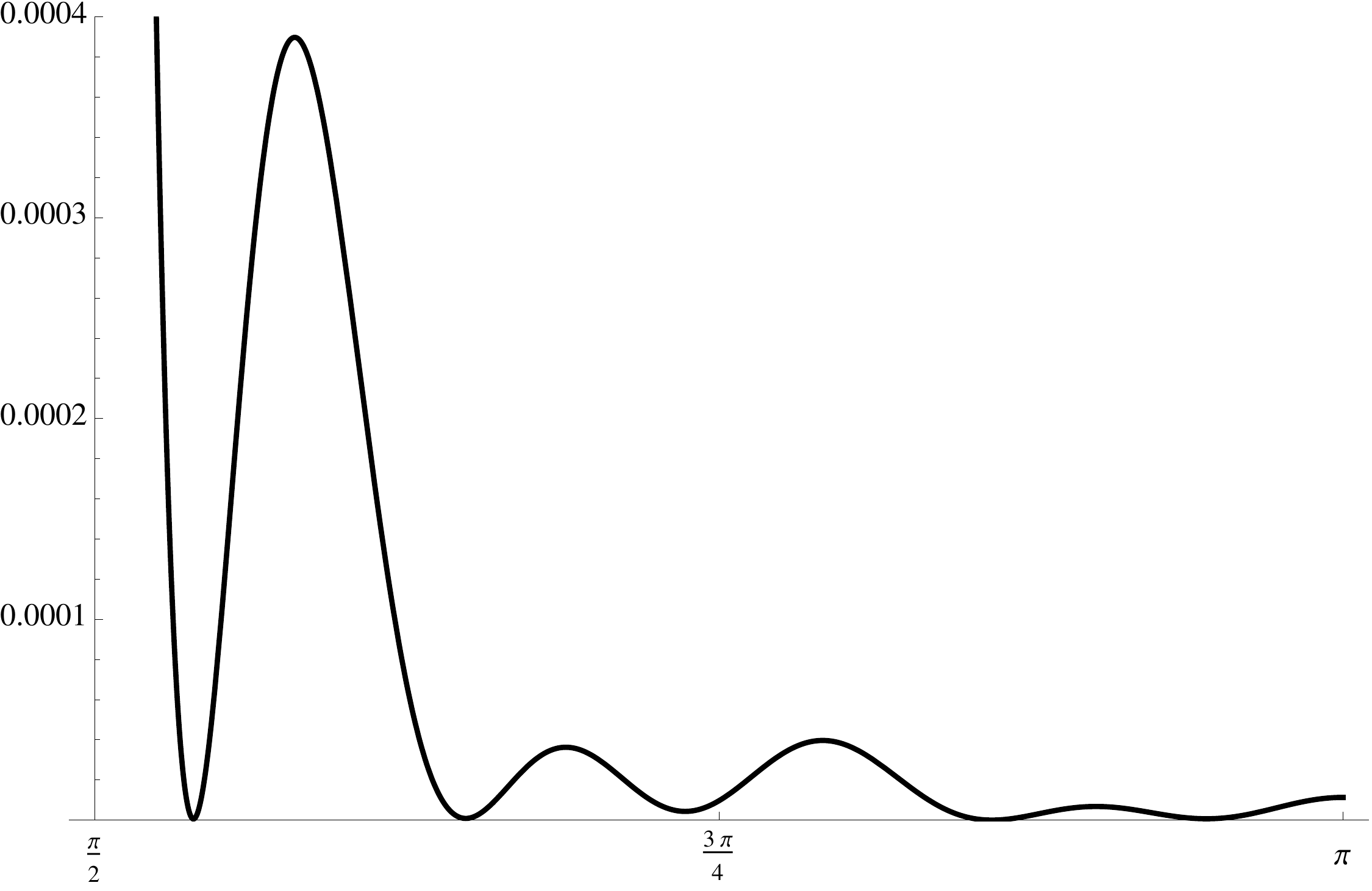}
\end{figure}

\subsection{Potential improvements}\label{subsectionHorizon}

While it appears that the scope for selecting an improved trigonometric polynomial may be limited, we see that the height $T_0$ to which the Riemann hypothesis has been verified has some influence on the width of the zero-free region that one can obtain by using this method.
We might ask then how much improvement we might expect in $R_0$ from this method as the parameter $T_0$ increases.
For example, if $T_0=3\cdot10^{11}$, then by choosing $\theta=1.85567$ with $F_{16}(\varphi)$, we obtain that $R_0=5.5666305$ is permissible.\footnote{The value $T_{0} = 3\cdot10^{11}$
has been announced by Jan B\"{u}the and Jens Franke in a personal communication.}

We also applied this method for $150$ values for $T_0$ up to approximately $10^{300}$, using $F_{16}(\varphi)$ as the trigonometric polynomial, with $t_0=10^5$, $r=5$, and $\theta=1.8552$.
We found that the value of $R_0$ that one may obtain in this region is well-approximated by
\[
R_0 = 5.4912 + \frac{2.0185}{\log T_0},
\]
suggesting that, as larger values of $T_0$ are established, the potential improvement in the width of the classical zero-free region that one may obtain by using this method is limited.

\section{Applications}\label{sectionApplications}

Theorem~\ref{thmR0} has a number of applications.
We first describe two improvements concerning the error term in the prime number theorem that follow almost immediately from this result.
\begin{cor}\label{TSwan}
Let 
\begin{equation*}\label{ep2}
\epsilon_{0}(x) = \sqrt{\frac{8}{17 \pi}} X^{1/2}e^{-X}, \quad X = \sqrt{(\log x)/R}, \quad R = 6.315.
\end{equation*}
Then
\begin{equation*}
\begin{split}
|\theta(x) - x| &\leq x \epsilon_{0}(x), \quad x \geq 149,\\
|\psi(x) - x| &\leq x \epsilon_{0}(x), \quad x \geq 23.
\end{split}
\end{equation*}
\end{cor}
\begin{proof}
Theorem~\ref{thmR0} can be used, as in \cite[p.\ 2]{Trudgianprime}, to show that there are no zeros of the Riemann zeta-function in the region 
\begin{equation}\label{zero2}
\sigma \geq 1 - \frac{1}{6.315 \log|t/17|}, \quad t\geq 24.
\end{equation}
The statement then follows from the arguments in \cite{Trudgianprime}.
\end{proof}
Here, the value $6.315$ replaces the value $6.455$ from \cite{Trudgianprime}.
The same substitution occurs in the following improvement on a result from this same article, where in addition the leading constant term is sharpened here from $0.2795$ to $0.2593$.
\begin{cor}\label{TLi}
If $x\geq 229$ then
\begin{equation*}
|\pi(x) - \li(x)| \leq 0.2593 \frac{x}{(\log x)^{3/4}} \exp\left(-\sqrt{\frac{\log x}{6.315}}\right).
\end{equation*}
\end{cor}
\begin{proof}
The proof is almost identical to that in \cite{Trudgianprime}.
We make use of the more sophisticated result from \cite{PlattTrudgianTheta} that $\theta(t) < t$ for all $t< 1.39\times 10^{17}$, and the proof follows as in \cite{Trudgianprime} by choosing $\alpha = 1.70$ and $x_{0} = 1.5\times 10^{8}$.
\end{proof}

It is worthwhile to note that a zero-free region of the form
\begin{equation}\label{parma}
\sigma \geq 1- \frac{1}{R \log(t/B)},
\end{equation}
for $t\geq 2$, where $B\geq1$ is a fixed constant, is asymptotically equivalent to the form analyzed in this paper.
Were one to prove (\ref{parma}) for some $B>1$ with $R=R_0$, for the value of $R_0$ derived in this paper, then one could further improve on Corollaries~\ref{TSwan} and~\ref{TLi}.

Finally, we very briefly note two additional applications of Theorem~\ref{thmR0}.
First, Faber and Kadiri \cite{Faber}, continuing the work of Schoenfeld \cite{Schoenfeld}, have used the width of this zero-free region \textit{inter alia} to obtain good bounds for the Chebyshev functions $\theta(x) = \sum_{p\leq x} \log p$ and $\psi(x) = \sum_{n\leq x} \Lambda(n)$, where $\Lambda(n)$ is von Mangoldt's function.
Second, Ramar\'{e} \cite[Thm.\ 1.1]{RamMob} employs Kadiri's value of $R_0$ when bounding $\sum_{n\leq x} \Lambda(n)/n$.
These results may be updated by using the smaller value of $R_0$ from Theorem~\ref{thmR0}.

\section*{Acknowledgements}

We thank the Institute for Computational and Experimental Research in Mathematics and the Center for Computation and Visualization at Brown University for computational resources.

\bibliographystyle{plain}
\bibliography{themastercanada}

\end{document}